\documentclass[a4paper,10pt]{article}
\usepackage{amsmath}
\usepackage{amssymb}
\usepackage{graphicx}
\usepackage{amsthm}
\usepackage{bbm}
\usepackage{mathrsfs}

\newtheoremstyle{mytheorem}{}{}{\itshape}{}{\bfseries}{:}{\newline}{}
\newtheoremstyle{mydefinition}{}{}{}{}{\bfseries}{:}{\newline}{}
\newtheoremstyle{myproof}{}{}{}{}{\bfseries}{:}{\newline}{#1#3}

\theoremstyle{mytheorem}
\newtheorem{thm}{Theorem}[]

\newtheorem{cor}[thm]{Corollary}

\newtheorem{prop}[thm]{Proposition}

\theoremstyle{mydefinition}

\theoremstyle{myproof}
\newtheorem{prf}{Proof}
\newtheorem{rmk}{Remarks}

\newcommand{\Qb}{\mathbb{Q}}
\newcommand{\Pb}{\mathbb{P}}
\newcommand{\Rb}{\mathbb{R}}
\newcommand{\Pt}{\tilde{\mathbb{P}}}
\newcommand{\Qt}{\tilde{\mathbb{Q}}}
\newcommand{\Fg}{\mathcal{F}}
\newcommand{\Gg}{\mathcal{G}}
\newcommand{\Ft}{\tilde{\mathcal{F}}}

\newcommand{\hs}{\hspace{2mm}}
\newcommand{\hsl}{\hspace{1mm}}
\newcommand{\ind}{\mathbbm{1}}
\renewcommand{\P}{\mathbb{P}}

\newcommand{\E}{\mathbb{E}}
\DeclareMathOperator{\spine}{spine}
\newcommand{\Et}{\tilde{\mathbb{E}}}

\author{Simon C.~Harris\footnote{Department of Mathematical Sciences, University of Bath, Bath, BA2 7AY, UK. email: \texttt{S.C.Harris@bath.ac.uk}} 
\  and 
Matthew I.~Roberts\footnote{Department of Statistics, University of Warwick, Coventry, CV4 7AL, UK. email: \texttt{mattiroberts@gmail.com}}}
\title{A strong law of large numbers for branching processes: almost sure spine events}

\begin{document}

\maketitle

\begin{abstract}
We demonstrate a novel strong law of large numbers for branching processes, with a simple proof via measure-theoretic manipulations and spine theory. Roughly speaking, any sequence of events that eventually occurs almost surely for the spine entails the almost sure convergence of a certain sum over particles in the population.
\end{abstract}

\section{Introduction}\label{intro_sec}
We shall work with a fairly general Markov branching process. To define this process, we suppose that we are given three ingredients:
\begin{itemize}
\item A Markov process $\psi_t$, $t\geq0$, in a measurable space $(J,\mathcal B)$;
\item A measurable function $R:J\to[0,\infty)$;
\item A collection of random variables $A(x)$, $x\in J$ taking values in $\{0,1,2,\ldots\}$, such that $M(x):= E[A(x)]-1 <\infty$.
\end{itemize}
Our branching process is then defined, under a probability measure $\Pb$, as follows: we begin with one particle. This particle moves around in $J$ like a copy of the process $\psi_t$. When at position $x$, it dies at rate $R(x)$, that is, if $\emptyset$ is our original particle, $X_\emptyset(t)$ is its position at time $t$ and $\tau_\emptyset$ is its time of death, then
\[\Pb(\tau_\emptyset > t \hsl | \hsl X_\emptyset(s),\hsl s\leq t) = \exp\left(-\int_0^t R(X_\emptyset(s)) ds\right).\]
At its time of death, it is replaced in its position $x$ by a random number of children, the number being specified by a copy of $A(x)$. These children then each independently show the same stochastic behaviour as their parent, moving around like independent copies of $\psi_t$ and branching at rate $R(x)$ when at position $x$ into a random number of particles that is an independent copy of $A(x)$. We let $N(t)$ be the set of all particles that are alive at time $t$; if $v\in N(t)$ then we let $X_v(t)$ be the position of particle $v$ at time $t$; and we let $A_v$ be the number of children of particle $v$.

We let $\Fg_t$, $t\geq0$ be the natural filtration of this process. We now extend our probability measure $\Pb$ to a new probability measure $\Pt$ on a bigger space by choosing one special line of descent which we call the \emph{spine}. The initial particle is part of the spine, and when a spine particle dies the new spine particle is chosen uniformly from amongst its children. We let the natural filtration of the new process, in which there is a branching process with one marked line of descent, be $\Ft_t$, $t\geq0$. Let $\xi_t$ be the position of the spine particle at time $t$, and let $\spine(t)$ be the set of particles that have been in the spine up to time $t$.

For details of all of the above, see \cite{hardy_harris:spine_approach_applications} or Chapter 2 of \cite{roberts:thesis}.

Suppose that $\zeta(t)$ is a non-negative martingale with respect to the filtration $\Gg_t:=\sigma(\xi_s,\hsl s\leq t)$, such that $\Et[\zeta(t)]=1$. We may write
\[\zeta(t) = \sum_{v\in N(t)}\zeta_v(t)\ind_{\{\xi_t = v\}}\]
where each $\zeta_v(t)$ is an $\Fg_t$-measurable random variable (see page 24 of \cite{roberts:thesis} for a proof).
Then
\[\tilde\zeta(t):=e^{-\int_0^t M(\xi_s)R(\xi_s) ds} \zeta(t) \prod_{v\in\spine(t)}(1+A_v)\]
is a martingale with respect to $\Ft_t$ (see Theorem 2.4 of \cite{roberts:thesis}). We define a new measure $\Qt$ by setting
\[\left.\frac{d\Qt}{d\Pt}\right|_{\Ft_t} := \tilde\zeta(t).\]
The measure $\Qt$ has a nice description in terms of the spine, although this will not be used in this article. Briefly, the motion of the spine is biased by the martingale $\zeta(t)$; branching events along the spine occur at an accelerated rate $(1+M(\xi_t))R(\xi_t)$ when the spine is at position $\xi_t$; and the number of children of the spine is size-biased. All other (non-spine) particles, once born, remain unaffected.

We also let $\Qb:=\Qt|_{\Fg_t}$ be a measure on $\Fg_t$, the natural filtration of the original branching process. Then
\[\left.\frac{d\Qb}{d\Pb}\right|_{\Fg_t} = \sum_{v\in N(t)}e^{-\int_0^t M(X_v(s))R(X_v(s)) ds} \zeta_v(t) =: Z(t)\]
and $Z(t)$ is a $\Pb$-martingale with respect to $\Fg_t$ (again see Theorem 2.4 of \cite{roberts:thesis} for details). Since $Z(t)$ is a positive martingale, it converges $\Pb$-almost surely to $Z(\infty):=\liminf Z(t)$.

We now state our main result. Suppose that $f(t)$ is $\Ft_t$-measurable for each $t$. Then, again, we may write each $f(t)$ via the representation
\[f(t) = \sum_{u\in N_t} f_u(t) \ind_{\{\xi_t = u\}}\]
where $f_u(t)$ is $\Fg_t$-measurable for each $t\geq0$ and each $u\in N(t)$.

\begin{thm}\label{spinecor}
Suppose that $\{ f(t) : t\geq0\}$ is $\Qt$-uniformly integrable. If $f(t)\to f$ $\Qt$-almost surely as $t\to\infty$ then
\begin{equation}\label{thmstatement}
\sum_{u\in N_t} f_u(t) \frac{e^{-\int_0^t M(X_u(s))R(X_u(s))ds} \zeta_u(t)}{Z(t)} \to \Qt[f|\Fg_\infty] \tag{$\star$}
\end{equation}
$\Qb$-almost surely. Furthermore, $\P\big((\star)\big|Z(\infty)>0\big)=1$.
\end{thm}

\begin{rmk}
\begin{enumerate}
\item Since $1/Z(t)$ is a positive $\Qb$-supermartingale (and thus converges almost surely to an almost surely finite limit), $Z(t)\to Z(\infty)$ $\Qb$-almost surely. Thus we may deduce from $(\star)$ that
\[\sum_{u\in N_t} f_u(t) e^{-\int_0^t M(X_u(s))R(X_u(s))ds} \zeta_u(t) \to \Qt[f|\Fg_\infty]Z(\infty) \hs \Qb\hbox{-almost surely.}\]
In fact under fairly mild conditions on the branching distributions $A(x)$, we have $\Qb(Z(\infty)<\infty)=1$, in which case we do not lose anything by rewriting $(\star)$ in this way.
\item In many cases of interest the events $\{Z(\infty)=0\}$ and $\{\exists t\in[0,\infty) : Z(t)=0\}$ agree to within a set of zero $\P$-probability. Then, of course,
\[\sum_{u\in N_t} f_u(t) e^{-\int_0^t M(X_u(s))R(X_u(s))ds} \zeta_u(t) \to \Qt[f|\Fg_\infty]Z(\infty) \hs \Pb\hbox{-almost surely.}\]

\end{enumerate}
\end{rmk}

\section{Some example applications}

%Theorem \ref{spinecor} has already been used in \cite{harris_hesse_kyprianou:bbm_strip} which considers near critical behaviour of a branching Brownian motion on a strip with killing at the boundary. 
%We outline here two more short examples showing how our strong law can be applied. The first example is folklore in branching processes, but we are not aware of another proof.

We outline here two examples showing how our strong law can be applied. The first example is folklore in branching processes, but we are not aware of another proof.
Theorem \ref{spinecor} has also been used in \cite{harris_hesse_kyprianou:bbm_strip} which considers a branching Brownian motion with killing on the boundary of a strip near criticality. 

\vspace{3mm}

\noindent
\textbf{In branching processes branching at rate $\beta$ into on average $m$ offspring, most particles branch at rate $m\beta$}.\\
Take a continuous-time branching process with constant birth rate $R(x)\equiv\beta$ and birth distribution $A(x)\equiv A$ satisfying $\E[A\log_+ A]<\infty$ with $m:=\E[A]$. Let $\zeta(t)\equiv1$. 
For any $\varepsilon>0$ we may take $f(t)=\ind_{\{|n_t/t-m\beta|<\varepsilon\}}$, the indicator that birth rate along the spine up to time $t$ is close to its expected value under $\Qt$, $m\beta$.
Then for any $\varepsilon>0$, $f(t)$ converges $\Qt$-almost surely to $1$. Thus Theorem \ref{spinecor}, together with some classical results on branching processes concerning the martingale $e^{-(m-1)\beta t} |N(t)|$, tells us that on the event that the process survives,
\[\frac{1}{|N(t)|}\sum_{u\in N(t)} f_u(t) \to 1 \hs\hs \Pb\text{-almost surely.}\]
This may be interpreted as saying that if we choose a particle uniformly at random from those alive at a large time $t$, and look at its history, we are likely to see that its average birth rate has been approximately $m\beta$. In particular, with binary branching, we see an average birth rate of $2\beta$ in typical particles (rather than $\beta$, which one might naively expect).

\vspace{3mm}

Our second example shows how the spatial behaviour of the spine can also be passed to other particles: if the spine shows ergodic behaviour, then so do many other particles.

\noindent
\textbf{Occupation densities and ergodic spines.}\\
Suppose that the motion of the spine $(\xi_t, t\geq0)$ is ergodic under $\Qt$ with invariant probability density $\pi$ in the sense that there exists some suitable class of functions $\mathcal{H}$ such that for any $h\in\mathcal H$,
\[\frac1t \int_0^t h(\xi_s)ds\to L_h:=\int_\Rb h(x)\pi(x)dx \hs\hs \Qt\hbox{-almost surely.}\]
Then for any continuous function $g:\Rb\to\Rb$ and any $h\in \mathcal H$,
\[\frac{1}{Z(t)}\sum_{u\in N(t)} g\left(\frac1t \int_0^t h(X_u(s)) ds\right) e^{-\int_0^t M(X_u(s))R(X_u(s)) ds}\zeta_u(t) \to g(L_h)\]
$\Qb$-almost surely. The same holds under $\Pb$ on the event $Z(\infty)>0$, which is one exposition of the general principle that if forcing the spine to show certain behaviour does not cause the corresponding martingale to disappear, then that behaviour appears in the original process.

\section{Measure theoretic results}
To prove Theorem \ref{spinecor} we need some simple measure theory. For this section we forget the branching setup and take any filtered probability space $(\Omega,\Fg,\Fg_t,P)$ and define $\Fg_\infty:=\bigvee_{t\geq0}\Fg_t$. Suppose that $X_t$, $t\geq0$ is a process such that $(E[X_t|\Fg_t], t\geq0)$ is almost surely c\`adl\`ag.

\noindent
\begin{prop}\label{asprop}
If
\[E[X_t|\Fg_\infty]\to Y \hs \hbox{almost surely,}\]
then
\[E[X_t|\Fg_t] \to Y \hs \hbox{almost surely.}\]
\end{prop}

\begin{prf}
Fix $\varepsilon>0$. We show that there exists an almost surely finite random variable $T$ such that
\[\sup_{t\geq T}E[X_t | \Fg_t] \leq Y + \varepsilon \hs \hbox{almost surely.}\]
By the c\`adl\`ag property, it is sufficient to take the supremum above over rationals greater than $T$; from now on all our suprema will be over rationals.

Since $E[X_t|\Fg_\infty]\to Y$, there exists an almost surely finite random variable $T_1$ such that
\[\sup_{t\geq T_1}E[X_t | \Fg_\infty] < Y + \varepsilon/2 \hs \hbox{almost surely,}\]
and since (by the fact that it is a closed martingale) $E[Y|\Fg_t]\to E[Y|\Fg_\infty]=Y$ ($Y$ is $\Fg_\infty$-measurable since it is the limit of $\Fg_\infty$-measurable random variables), there exists an almost surely finite random variable $T_2$ such that
\[\sup_{t\geq T_2}E[Y|\Fg_t] < Y + \varepsilon/2 \hs \hbox{almost surely.}\]
Let $T=T_1 \vee T_2$. Then
\begin{eqnarray*}
\sup_{t\geq T} E[X_t|\Fg_t] &\leq& \sup_{t\geq T_2} \sup_{s\geq T_1} E[X_s|\Fg_t]\\
&=& \sup_{t\geq T_2} \sup_{s\geq T_1} E[E[X_s|\Fg_\infty]|\Fg_t]\\
&\leq& \sup_{t\geq T_2} E\left[\left. \sup_{s\geq T_1} E[X_s|\Fg_\infty] \right|\Fg_t\right]\\
&\leq& \sup_{t\geq T_2} E[ Y + \varepsilon/2 | \Fg_t ]\\
&\leq& Y + \varepsilon
\end{eqnarray*}
(all statements hold almost surely). Thus $\limsup E[X_t|\Fg_t]\leq Y$; the proof that $\liminf E[X_t|\Fg_t]\geq Y$ is similar.\qed
\end{prf}

\begin{cor}\label{ascorol}
Suppose that the collection of random variables $\{X_t, t\geq0\}$ is uniformly integrable. If
\[X_t \to X \hs \hbox{almost surely}\]
then
\[E[X_t | \Fg_t] \to E[X|\Fg_\infty] \hs \hbox{almost surely.}\]
\end{cor}

\begin{prf}
Let $Y=E[X|\Fg_\infty]$; then by uniform integrability,
\[E[X_t|\Fg_\infty]\to Y \hs \hbox{almost surely.}\]
Proposition 1 now gives the result.\qed
\end{prf}

\section{The proof of Theorem \ref{spinecor}}
We now return to the notation from Section \ref{intro_sec}.

\begin{proof}[Proof of Theorem \ref{spinecor}]
We recall Theorem 8.2 of Hardy and Harris \cite{hardy_harris:spine_approach_applications}, which says that under the conditions above,
\[\Qt[f(t) | \Fg_t] = \sum_{u\in N_t} f_u(t) \frac{e^{-\int_0^t M(X_u(s))R(X_u(s))ds} \zeta_u(t)}{Z(t)}.\]
Now if $f(t)$ converges $\Qt$-almost surely to $f$ then by Corollary \ref{ascorol} we have
\[\Qt[f(t) | \Fg_t] \to \Qt[f|\Fg_\infty] \hs \Qt\hbox{-almost surely}\]
and hence
\[\sum_{u\in N_t} f_u(t) \frac{e^{-\int_0^t M(X_u(s))R(X_u(s))ds} \zeta_u(t)}{Z(t)} \to \Qt[f|\Fg_\infty] \hs \Qt\hbox{-almost surely.}\]
Finally, for any ($\Fg_\infty$-measurable) event $A$ such that $\Qb(A)=1$,
\begin{align*}
\P(A|Z(\infty)>0) &= \frac{\P(A\cap\{Z(\infty)>0\})}{\P(Z(\infty)>0)}\\
&= \frac{\Qb\left[\frac{1}{Z(\infty)}\ind_{A\cap\{Z(\infty)>0\}}\right]}{\P(Z(\infty)>0)}\\
&= \frac{\Qb\left[\frac{1}{Z(\infty)}\ind_{\{Z(\infty)>0\}}\right]}{\P(Z(\infty)>0)}\\
&= 1.\qedhere
\end{align*}
\end{proof}

\bibliographystyle{plain}

\def\cprime{$'$}

\end{document}